\documentclass[10pt]{amsart}

\usepackage{latexsym}
\usepackage{amsmath}
\usepackage{amssymb}
\usepackage[all]{xy}
\usepackage{eucal}
\usepackage{multicol}

\newtheorem{theorem}{Theorem }[section]
\newtheorem{proposition}[theorem]{Proposition}

\newtheorem{definition}[theorem]{Definition}

\newtheorem{example}[theorem]{Example}
\newtheorem{remark}[theorem]{Remark}

\begin{document}
\begin{flushright}
 [BoundingProblem(2011-06-17)]\\
 \end{flushright}

\title{On  Bounding Problems in Totally Ordered Commutative Semi-Groups} 
\author{Susumu Oda}

\footnote[0]{
\hspace{10mm} Keywords:\ ascending chain, estimate problem } 

\maketitle

 \baselineskip=16pt

\begin{center} Dedicated to Unyo Oda, who had lived for only ten years in Kochi
\end{center}

\begin{abstract} 
 The following is shown : Let  $S=\{a_1,a_2,..,a_{2n}\}$ be  a subset of a totally ordered commutative semi-group $(G,*,\leq)$ with $a_1\leq a_2\leq \cdots \leq a_{2n}$. Provided that a system of $n$ $a_{i_k} * a_{j_k}\ (a_{i_k}, a_{j_k} \in G ;\ 1 \leq k \leq n)$, where all $2n$ elements in $S$  must be used, are less than an element $N\ (\in G)$, then $a_1*a_{2n}, a_2*a_{2n-1}, \ldots, a_n*a_{n+1}$ are all less than $N$. This may be called the Upper Bounding Case. Moreover in the same way, we shall treat also the Lower Bounding Case. 
\end{abstract}

\vspace*{15mm}

\section{The Early Stage of a Problem}

\bigskip

The author was asked  one question (Proposition \ref{PP1} below), which is easily understandable for  almost people.
 
In Mathematical World, what almost all people easily understand often seems to be  difficult to be proved.
 We suspect that this may probably be such one.

But it seems that this is not the case. The answer is affirmative as is seen in the next section.
 However, after finished finding a proof to it, the author comes to be afraid that it is really of value.  He only hope that it is interesting to some people. In addition, he has heard that this result is available for proving another result in Mathematics. 

\bigskip

{\sl The proofs of Propositions in this section will be given in the section 2.}

\begin{proposition}{{\rm (Upper Bounding.)}}\label{PP1}
 Let  $S=\{a_1,a_2,..,a_{2n}\}$ be  a set of positive integers with
$a_1\leq a_2\leq \cdots \leq a_{2n}$. 
Provided once that a system of inequalities: 
$$
(1_n)\left\{ \begin{array}{ccc}
 a_{i_1}+a_{j_1}&<&N\\
a_{i_2}+a_{j_2}&<&N\\
\ldots&&\\
\ldots&&\\
a_{i_n}+a_{j_n}&<&N\\
\end{array}\right.
$$
 for an integer $N\in \mathbb{R}$ and that $\{{i_1},{j_1},{i_2},{j_2},\ldots ,{i_n},{j_n}\}=\{ 1,2,3,\ldots,2n \}$,
then  the following system of special inequalities hold:
$$
(2_n)\left\{ \begin{array}{ccc}
a_1 + a_{2n} &<& N\\
a_2 + a_{2n-1} &<& N\\
 \cdots &&\\
 \ldots &&\\
a_n + a_{n+1} &<& N.
\end{array}\right.
$$
 \end{proposition}

\begin{remark} {\rm In Proposition \ref{PP1}, all of elements in $S=\{a_1,a_2,..,a_{2n}\}$  are not necessarily positive  because we have only consider $0<a_1+c,a_2+c,..,a_{2n}+c$ and $N+c$ instead of $a_1,a_2,..,a_{2n}$ by choosing some $c \in \mathbb{N}$.}
\end{remark}

\bigskip

\begin{proposition}{{\rm (Lower Bounding.)}}\label{PP2}
 Let  $S=\{a_1,a_2,..,a_{2n}\}$ be  a set of positive real numbers with
$a_1\geq a_2\geq \cdots \geq a_{2n}$. 
Provided once a system of inequalities: 
$$
({}_n1)\left\{ \begin{array}{ccc}
 a_{i_1}+a_{j_1}&>&N\\
a_{i_2}+a_{j_2}&>&N\\
\ldots&&\\
\ldots&&\\
a_{i_n}+a_{j_n}&>&N\\
\end{array}\right.
$$
 for a real number $N\in \mathbb{R}$ and  $\{{i_1},{j_1},{i_2},{j_2},\ldots ,{i_n},{j_n}\}=\{ 1,2,3,\ldots,2n \}$,
then  the following system of special inequalities hold:
$$
({}_n2)\left\{ \begin{array}{ccc}
a_1 + a_{2n} &>& N\\
a_2 + a_{2n-1} &>& N\\
 \cdots &&\\
 \ldots &&\\
a_n + a_{n+1} &>& N.
\end{array}\right.
$$
 \end{proposition}

\bigskip
\bigskip
\bigskip

\section{Main Results}

In this section, we generalize the results in the section 1 and give their proofs.

 By these generalizations, we are able to consider ``product" instead of ``sum" and to obtain the similar results for ``product". 

\begin{definition}{\rm (cf.[1,p.153])} {\rm 
 Let $(G,*)$ be an commutative semi-group ({\it e.i.}, a commutative semi-group with  a total order $\leq$.
 If for any $\alpha, \beta, \gamma, \delta \in G$, the following conditions hold:
$$\alpha \leq \beta, \gamma \leq \delta\ \ \Rightarrow \alpha *\gamma \leq \beta * \delta, $$ 
 then $(G,*,\leq)$ (abbreviated by $G$ if no confusions do not happen) is called {\it a totally ordered commutative semi-group}}.
\end{definition}

\bigskip

\begin{theorem}{{\rm (Upper Bounding.)}}\label{T1}
 Let $(G,*,\leq)$ be a totally ordered commutative semi-group and let  $S=\{a_1,a_2,..,a_{2n}\}$ be  a finite subset of $G$  with  $a_1\leq a_2\leq \cdots \leq a_{2n}$ ($n \in \mathbb{N}$). 
Provided once that a system of inequalities: 
$$
(1_n)\left\{ \begin{array}{ccc}
 a_{i_1}*a_{j_1}&<&N\\
a_{i_2}*a_{j_2}&<&N\\
\ldots&&\\
\ldots&&\\
a_{i_n}*a_{j_n}&<&N\\
\end{array}\right.
$$
 for a  $N\in G$ and that $\{{i_1},{j_1},{i_2},{j_2},\ldots ,{i_n},{j_n}\}=\{ 1,2,3,\ldots,2n \}$,
then  the following system of special inequalities hold:
$$
(2_n)\left\{ \begin{array}{ccc}
a_1 * a_{2n} &<& N\\
a_2 * a_{2n-1} &<& N\\
 \cdots &&\\
 \ldots &&\\
a_n * a_{n+1} &<& N.
\end{array}\right.
$$
 \end{theorem}

\begin{proof}
 We show this by induction on the positive integer $n$.  
 
{\bf (i)} If the inequality $a_1 * a_{2n}<N$ appears in $(1_n)$, then remove $a_1 * a_{2n}<N$ in $(1_n)$, which is denoted by $(1_{n-1})$,  and consider $a_2\leq \cdots \leq a_{2n-1}$. By the induction hypothesis, we have  
$$
(2_{n-1})\left\{ \begin{array}{ccc}
a_2 * a_{2n-1} &<& N\\
a_3 * a_{2n-2} &<& N\\
 \cdots &&\\
 \ldots &&\\
a_{n} * a_{n+1} &<& N.
\end{array}\right.
$$
Then $(2_n)$ holds.

{\bf (ii)} Consider the case that the inequality $a_1 * a_{2n}<N$ does not appear in $(1_n)$.  We may assume that $i_1=1$ and $j_n=2n$.
 Then there exist $\ell$ and $\ell'\ (1<\ell, \ell'<2n)$ such that 
\begin{eqnarray*}
a_1 * a_\ell&<&N\\
a_{\ell'} * a_{2n}&<&N
 \end{eqnarray*}
appear in $(1_n)$.
 It follows that  $a_1 * a_{2n} \leq a_{\ell'} * a_{2n}<N$ because $a_1 \leq a_{\ell'}\ (1<\ell')$.

(ii-1) The case :  $a_\ell * a_{\ell'}\geq N$.

Since $a_{\ell} \leq a_{2n}$, it follows that $N\leq a_\ell * a_{\ell'}\leq a_{\ell'} * a_{2n}<N$, which is a contradiction. So this case does not occur.

(ii-2) The case : $a_\ell * a_{\ell'}<N$.

  Remove $a_1 * a_\ell<N$ and $a_{\ell'} * a_{2n}<N$ from $(1_n)$ and insert $a_\ell * a_{\ell'}<N$.  Then we  have the system of inequalities :
$$
(1'_{n-1})\left\{ \begin{array}{ccc}
a_{i_2} * a_{j_2}&<&N\\
\ldots&&\\
\ldots&&\\
a_{i_{n-1}} * a_{j_{n-1}}&<&N\\
a_\ell * a_{\ell'} &<&N,
\end{array}\right.
$$
where  $\{{i_2}, {j_2},\ldots ,{i_{n-1}}, {j_{n-1}}, \ell, {\ell'}\}=\{ 2, 3, \ldots, {2n-1}\}$, that is, all $a_2\leq \cdots \leq a_{2n-1}$ appear in $(1'_{n-1})$
 We can  apply the induction hypothesis to $(1'_{n-1})$, and obtain the same 
 $(2_{n-1})$ in (i) and hence $(2_n)$ holds  because $a_1 * a_{2n}<N$ as mentioned above.

Therefore the proof has finished.
\end{proof}

\begin{remark}
{\rm The proof of Proposition \ref{PP1} is obtained by replacing ``$*$" by ``$+$" in Theorem \ref{T1}.}
\end{remark}

\bigskip

By this generalization, we obtain the similar result for ``product" as follows.

\begin{proposition}\label{P2}
  Let  $S=\{a_1,a_2,..,a_{2n}\}$ be  a set of positive real numbers with
$a_1\leq a_2\leq \cdots \leq a_{2n}$. 
Provided once that a system of inequalities: 
$$
(1_n)\left\{ \begin{array}{ccc}
 a_{i_1} \cdot a_{j_1}&<&N\\
a_{i_2} \cdot a_{j_2}&<&N\\
\ldots&&\\
\ldots&&\\
a_{i_n} \cdot a_{j_n}&<&N\\
\end{array}\right.
$$
 for a real number $N\in \mathbb{R}$ and that $\{{i_1},{j_1},{i_2},{j_2},\ldots ,{i_n},{j_n}\}=\{ 1,2,3,\ldots,2n \}$,
then  the following system of special inequalities hold:
$$
(2_n)\left\{ \begin{array}{ccc}
a_1 \cdot a_{2n} &<& N\\
a_2 \cdot a_{2n-1} &<& N\\
 \cdots &&\\
 \ldots &&\\
a_n \cdot a_{n+1} &<& N.
\end{array}\right.
$$

\end{proposition}

\bigskip

Having a glimpse of the proof of Theorem \ref{T1}, we can assert and prove the lower bounding Theorem below.
 It may be obvious because we have only to give a symmetrical consideration.
 In other word, in Theorem \ref{T1} and its proof, we have only to replace $``<"$ and $``\leq"$  by $``>"$ and $``\geq"$, respectively.  The same arguments in its proof are effective.

\bigskip

\begin{theorem}{{\rm (Lower Bounding.)}}\label{T2}
 Let $(G,*,\leq)$ be a totally ordered additive semi-group and let  $S=\{a_1,a_2,..,a_{2n}\}$ be  a finite subset of $G$  with  $a_1\leq a_2\leq \cdots \leq a_{2n}$ ($n \in \mathbb{N}$).  
Provided once that a system of inequalities: 
$$
({}_n1)\left\{ \begin{array}{ccc}
 a_{i_1} * a_{j_1}&>&N\\
a_{i_2} * a_{j_2}&>&N\\
\ldots&&\\
\ldots&&\\
a_{i_n} * a_{j_n}&>&N\\
\end{array}\right.
$$
 for an element $N\in \mathbb{G}$ and  $\{{i_1},{j_1},{i_2},{j_2},\ldots ,{i_n},{j_n}\}=\{ 1,2,3,\ldots,2n \}$,
then  the following system of special inequalities hold:
$$
({}_n2)\left\{ \begin{array}{ccc}
a_1 * a_{2n} &>& N\\
a_2 * a_{2n-1} &>& N\\
 \cdots &&\\
 \ldots &&\\
a_n * a_{n+1} &>& N.
\end{array}\right.
$$
 \end{theorem}

\begin{remark}
 {\rm The proof of Proposition \ref{PP2} is obtained by replacing ``$*$" by ``$+$" in Theorem \ref{T2}.}
\end{remark}

\begin{proposition}\label{P4}
 $S=\{a_1,a_2,..,a_{2n}\}$ be  a set of positive real numbers with
$a_1\leq a_2\leq \cdots \leq a_{2n}$. 
Provided once a system of inequalities: 
$$
({}_n1)\left\{ \begin{array}{ccc}
 a_{i_1} \cdot a_{j_1}&>&N\\
a_{i_2} \cdot a_{j_2}&>&N\\
\ldots&&\\
\ldots&&\\
a_{i_n} \cdot a_{j_n}&>&N\\
\end{array}\right.
$$
 for a real number $N\in \mathbb{R}$ and  $\{{i_1},{j_1},{i_2},{j_2},\ldots ,{i_n},{j_n}\}=\{ 1,2,3,\ldots,2n \}$$\{a_{i_1},a_{j_1},a_{i_2},a_{j_2},\ldots ,a_{i_n},a_{j_n}\}=S$,
then  the following system of special inequalities hold:
$$
({}_n2)\left\{ \begin{array}{ccc}
a_1 \cdot a_{2n} &>& N\\
a_2 \cdot a_{2n-1} &>& N\\
 \cdots &&\\
 \ldots &&\\
a_n \cdot a_{n+1} &>& N.
\end{array}\right.
$$
\end{proposition}


\bigskip

The Propositions considered above are only some of the precise expressions (1) $\sim$ (6) in the following Examples.

\begin{example}\label{E1}
{\rm  We can consider the following totally ordered commutative semi-group as $G$ in Theorems \ref{T1} and \ref{T2}:\\
(1) $(\mathbb{Z}_{\geq 1},+,\leq)$, and $(\mathbb{Z}_{>0},\times,\leq)$,\\
(2) $(\mathbb{Z},+,\leq)$,\\
(3) $(\mathbb{Q}_{> 0},+,\leq)$, and $(\mathbb{Q}_{>0},\times,\leq)$\\
(4) $(\mathbb{Q},+,\leq)$,\\
(5) $(\mathbb{R}_{> 0},+,\leq)$, and $(\mathbb{R}_{>0},\times,\leq)$\\
(6) $(\mathbb{R},+,\leq)$,\\
(7) ${(\mathbb{Z}_{\geq 0}}^n,+,\leq_{lex})$, where $\leq_{lex}$ denotes the lexicographic order in ${\mathbb{Z}_{\geq 0}}^n$.
 Note that this $({\mathbb{Z}_{\geq 0}}^n,+,\leq_{lex})$ can be used for Theory of G\"{o}bner Basis concerning  polynomial rings.} 
\end{example}

\bigskip
\bigskip

\section{Examples}

It is easy to know that the numbers of systems $(1_n)$ is equal to $\Pi_{k=1}^n(2k-1)$ up to symmetry. For $n=3$, there exist $15$ inequalities in $(1_2)$, and for $n=4$, there exist  $105$ inequalities in $(1_4)$ when each $a_i$'s values are fixed.  The author has checked several $n=4$ case concretely.  Here we show  three cases for $n=3$. Readers could check our results are true for such cases. 

\bigskip

\noindent
{\bf Example 0-1(Trivial Example)}

$ 2n=6 : a_1=1, a_2=2, a_3=3, a_4=4, a_5=5, a_6=6$

$$
\begin{array}{|r|r|r|r|r|r|}
 \hline
{\rm (sum)\ +}&   &      &   &{\rm Max}&{\rm Min}\\
 \hline
1_3(1)& 1 + 2 = 3& 3 + 4 = 7& 5 + 6 = 11& 11& 3\\
 \hline
 1_3(2)& 1 + 2 = 3& 3 + 5 = 8& 4 + 6 = 10& 10& 3\\
 \hline
 1_3(3)& 1 + 2 = 3& 3 + 6 = 9& 4 + 5 = 9& 9& 3\\
 \hline
1_3(4)& 1 + 3 = 4& 2 + 4 = 6& 5 + 6 = 11& 11& 4\\
\hline
1_3(5)& 1 + 3 = 4& 2 + 5 = 7& 4 + 6 = 10& 10& 4\\
\hline
1_2(6)& 1 + 3 = 4& 2 + 6 = 8& 4 + 5 = 9& 9& 4\\
\hline
1_3(7)& 1 + 4 = 5& 2 + 3 = 5& 5 + 6 = 11& 11& 5\\
\hline
1_3(8)& 1 + 4 = 5& 2 + 5 = 7& 3 + 6 = 9& 9& 5\\
\hline
1_3(9)& 1 + 4 = 5& 2 + 6 = 8& 3 + 5 = 8& 8& 5\\
\hline
1_3(10)& 1 + 5 = 6& 2 + 3 = 5& 4 + 6 = 10& 10&  5\\
\hline
1_3(11)& 1 + 5 = 6& 2 + 4 = 6& 3 + 6 = 9& 9& 6\\
\hline
1_3(12)& 1 + 5 = 6& 2 + 6 = 8& 3 + 4 = 7& 8& 6\\
\hline
1_3(13)& 1 + 6 = 7& 2 + 3 = 5& 4 + 5 = 9& 9& 5\\
\hline
1_3(14)& 1 + 6 = 7& 2 + 4 = 6& 3 + 5 = 8& 8& 6\\
\hline
(\#) 1_3(15)& 1 + 6 = 7& 2 + 5 = 7& 3 + 4 = 7& 7& 7\\
  \hline
\end{array}
$$
 
\vspace*{10mm}

\noindent
{\bf Example 0-2(Trivial Example)}

$2n=6 : a_=1, a_2=2, a_3=3, a_4=4, a_5=5, a_6=6$

$$
\begin{array}{|r|r|r|r|r|r|}
\hline
(product)&    &    &   & Max & Min \\
\hline
1_3(1)& 1 \cdot  2 = 2& 3 \cdot  4 = 12& 5 \cdot  6 = 30& 30& 2\\
 \hline
1_3(2)& 1 \cdot  2 = 2& 3 \cdot  5 = 15& 4 \cdot  6 = 24& 24& 2\\
 \hline
1_3(3)& 1 \cdot  2 = 2& 3 \cdot  6 = 18& 4 \cdot  5 = 20& 20& 2\\
 \hline
1_3(4)& 1 \cdot  3 = 3& 2 \cdot  4 = 8& 5 \cdot  6 = 30& 30& 3\\
 \hline
1_3(5)& 1 \cdot  3 = 3& 2 \cdot  5 = 10& 4 \cdot  6 = 24& 24& 3\\
 \hline
1_2(6)& 1 \cdot  3 = 3& 2 \cdot  6 = 12& 4 \cdot  5 = 20& 20& 3\\
 \hline
1_3(7)& 1 \cdot  4 = 4& 2 \cdot  3 = 6& 5 \cdot  6 = 30& 30& 4\\
 \hline
1_3(8)& 1 \cdot  4 = 4& 2 \cdot  5 = 10& 3 \cdot  6 = 18& 18& 4\\
 \hline
1_3(9)& 1 \cdot  4 = 4& 2 \cdot  6 = 12& 3 \cdot  5 = 15& 15& 4\\
 \hline
1_3(1)& 1 \cdot  5 = 5& 2 \cdot  3 = 6& 4 \cdot  6 = 24& 24& 5\\
 \hline
1_3(11)& 1 \cdot  5 = 5& 2 \cdot  4 = 8& 3 \cdot  6 = 18& 18& 5\\
 \hline
1_3(12)& 1 \cdot  5 = 5& 2 \cdot  6 = 12& 3 \cdot  4 = 12& 12& 5\\
 \hline
1_3(13)& 1 \cdot  6 = 6& 2 \cdot  3 = 6& 4 \cdot  5 = 20& 20& 6\\
 \hline
1_3(14)& 1 \cdot  6 = 6& 2 \cdot  4 = 8& 3 \cdot  5 = 15& 15& 6\\
 \hline
(\#) 1_3(15)& 1 \cdot  6& = 6 2 \cdot  5 = 10& 3 \cdot  4 = 12& 12& 6\\
  \hline
  \end{array}
$$

\vspace*{10mm}

\noindent
{\bf Example 1-1}

$2n=6 : a_1=1, a_2=3, a_3=6, a_4=8, a_5=9, a_6=11$

$$ 
 \begin{array}{|r|r|r|r|r|r|}
\hline
(sum)&           &     &    & Max& Min\\ 
\hline
1_3(1)& 1 + 3 = 4& 6 + 8 = 14& 9 + 11 = 20& 20& 4\\ 
 \hline
1_3(2)& 1 + 3 = 4& 6 + 9 = 15& 8 + 11 = 19& 19& 4\\
 \hline
1_3(3)& 1 + 3 = 4& 6 + 11 = 17& 8 + 9 = 17& 17& 4\\
 \hline
1_3(4)& 1 + 6 = 7& 3 + 8 = 11& 9 + 11 = 20& 20& 7\\
 \hline
1_3(5)& 1 + 6 = 7& 3 + 9 = 12& 8 + 11 = 19& 19& 7\\
 \hline
1_2(6)& 1 + 6 = 7& 3 + 11 = 14& 8 + 9 = 17& 17& 7\\
 \hline
1_3(7)& 1 + 8 = 9& 3 + 6 = 9& 9 + 11 = 20& 20& 9\\
 \hline
1_3(8)& 1 + 8 = 9& 3 + 9 = 12& 6 + 11 = 17& 17& 9\\
 \hline
1_3(9)& 1 + 8 = 9& 3 + 11 = 14& 6 + 9 = 15& 15& 9\\
 \hline
1_3(10)& 1 + 9 = 10& 3 + 6 = 9& 8 + 11 = 19& 19& 9\\
 \hline
1_3(11)& 1 + 9 = 10& 3 + 8 = 11& 6 + 11 = 17& 17& 10\\
 \hline
1_3(12)& 1 + 9 = 10& 3 + 11 = 14& 6 + 8 = 14& 14& 10\\
 \hline
1_3(13)& 1 + 11 = 12& 3 + 6 = 9& 8 + 9 = 17& 17& 9\\
 \hline
1_3(14)& 1 + 11 = 12& 3 + 8 = 11& 6 + 9 = 15& 15& 11\\
(\#) 1_3(15)& 1 + 11 = 12& 3 + 9 = 12& 6 + 8 = 14& 14& 12\\
  \hline
  \end{array}
$$

\vspace*{10mm}

\noindent
{\bf Example 1-2}

$2n=6 : a_1=1, a_2=3, a_3=6, a_4=8, a_5=9, a_6=11$

 $$
 \begin{array}{|r|r|r|r|r|r|}
\hline
(product)&           &     &    & Max& Min\\ 
\hline
1_3(1)& 1 \cdot  3 = 3& 6 \cdot  8 = 48& 9 \cdot  11 = 99& 99& 3\\ 
\hline
1_3(2)& 1 \cdot  3 = 3& 6 \cdot  9 = 54& 8 \cdot  11 = 88& 88& 3\\ 
\hline
1_3(3)& 1 \cdot  3 = 3& 6 \cdot  11 = 66& 8 \cdot  9 = 72& 72& 3\\
 \hline
1_3(4)& 1 \cdot  6 = 6& 3 \cdot  8 = 24& 9 \cdot  11 = 99& 99& 6\\
 \hline
1_3(5)& 1 \cdot  6 = 6& 3 \cdot  9 = 27& 8 \cdot  11 = 88& 88& 6\\ 
\hline
1_2(6)& 1 \cdot  6 = 6& 3 \cdot  11 = 33& 8 \cdot  9 = 72& 72& 6\\ 
\hline
1_3(7)& 1 \cdot  8 = 8& 3 \cdot  6 = 18& 9 \cdot  11 = 99& 99& 8\\ 
\hline
1_3(8)& 1 \cdot  8 = 8& 3 \cdot  9 = 27& 6 \cdot  11 = 66& 66& 8\\ 
\hline
1_3(9)& 1 \cdot  8 = 8& 3 \cdot  11 = 33& 6 \cdot  9 = 54& 54& 8\\ 
\hline
1_3(10)& 1 \cdot  9 = 9& 3 \cdot  6 = 18& 8 \cdot  11 = 88& 88& 9\\ 
\hline
1_3(11)& 1 \cdot  9 = 9& 3 \cdot  8 = 24& 6 \cdot  11 = 66& 66& 9\\ 
\hline
1_3(12)& 1 \cdot  9 = 9& 3 \cdot  11 = 33& 6 \cdot  8 = 48& 48& 9\\ 
\hline
1_3(13)& 1 \cdot  11 = 11& 3 \cdot  6 = 18& 8 \cdot  9 = 72& 72& 11\\ 
\hline
1_3(14)& 1 \cdot  11 = 11& 3 \cdot  8 = 24& 6 \cdot  9 = 54& 54& 11\\ 
\hline
(\#) 1_3(15)& 1 \cdot  11 = 11& 3 \cdot  9 = 27& 6 \cdot  8 = 48& 48& 11\\ 
\hline
  \end{array}
$$

\vspace*{10mm}

\noindent
{\bf Example 2-1}

$2n=6 : a_1=2, a_2=7, a_3=11, a_4=14, a_5=16, a_6=17$

 $$
 \begin{array}{|r|r|r|r|r|r|}
\hline
(sum)&           &             &            & Max& Min\\
 \hline
1_3(1)& 2 + 7 = 9& 11 + 14 = 25& 16 + 17 = 33& 33& 9\\
 \hline
1_3(2)& 2 + 7 = 9& 11 + 16 = 27& 14 + 17 = 31& 31& 9\\
 \hline
1_3(3)& 2 + 7 = 9& 11 + 17 = 28& 14 + 16 = 30& 30& 9\\
 \hline
1_3(4)& 2 + 11 = 13& 7 + 14 = 21& 16 + 17 = 33& 33& 13\\
 \hline
1_3(5)& 2 + 11 = 13& 7 + 16 = 23& 14 + 17 = 31& 31& 13\\
 \hline
1_2(6)& 2 + 11 = 13& 7 + 17 = 24& 14 + 16 = 30& 30& 13\\
 \hline
1_3(7)& 2 + 14 = 16& 7 + 11 = 18& 16 + 17 = 33& 33& 16\\
 \hline
1_3(8)& 2 + 14 = 16& 7 + 16 = 23& 11 + 17 = 28& 28& 16\\
 \hline
1_3(9)& 2 + 14 = 16& 7 + 17 = 24& 11 + 16 = 27& 27& 16\\
 \hline
1_3(10)& 2 + 16 = 18& 7 + 11 = 18& 14 + 17 = 31& 31& 18\\
 \hline
1_3(11)& 2 + 16 = 18& 7 + 14 = 21& 11 + 17 = 28& 28& 18\\
 \hline
1_3(12)& 2 + 16 = 18& 7 + 17 = 24& 11 + 14 = 25& 25& 18\\
 \hline
1_3(13)& 2 + 17 = 19& 7 + 11 = 18& 14 + 16 = 30& 30& 18\\
 \hline
1_3(14)& 2 + 17 = 19& 7 + 14 = 21& 11 + 16 = 27& 27& 19\\
 \hline
(\#) 1_3(15)& 2 + 17 = 19& 7 + 16 = 23& 11 + 14 = 25& 25& 19\\
  \hline
  \end{array}
$$

\vspace*{10mm}

\noindent
{\bf Example 2-2}

$2n=6 : a_1=2, a_2=7, a_3=11, a_4=14, a_5=16, a_6=17$

$$
  \begin{array}{|r|r|r|r|r|r|}
 \hline
(product)&            &                   &                  & Max& Min\\
 \hline
1_3(1)& 2 \cdot  7 = 14& 11 \cdot  14 = 154& 16 \cdot  17 = 272& 272& 14\\ 
\hline
1_3(2)& 2 \cdot  7 = 14& 11 \cdot  16 = 176& 14 \cdot  17 = 238& 238& 14\\
 \hline
1_3(3)& 2 \cdot  7 = 14& 11 \cdot  17 = 187& 14 \cdot  16 = 224& 224& 14\\
 \hline
1_3(4)& 2 \cdot  11 = 22& 7 \cdot  14 = 98& 16 \cdot  17 = 272& 272& 22\\
 \hline
1_3(5)& 2 \cdot  11 = 22& 7 \cdot  16 = 112& 14 \cdot  17 = 238& 238& 22\\
 \hline
1_2(6)& 2 \cdot  11 = 22& 7 \cdot  17 = 119& 14 \cdot  16 = 224& 224& 22\\
 \hline
1_3(7)& 2 \cdot  14 = 28& 7 \cdot  11 = 77& 16 \cdot  17 = 272& 272& 28\\
 \hline
1_3(8)& 2 \cdot  14 = 28& 7 \cdot  16 = 112& 11 \cdot  17 = 187& 187& 28\\
 \hline
1_3(9)& 2 \cdot  14 = 28& 7 \cdot  17 = 119& 11 \cdot  16 = 176& 176& 28\\
 \hline
1_3(10)& 2 \cdot  16 = 32& 7 \cdot  11 = 77& 14 \cdot  17 = 238& 238& 32\\
 \hline
1_3(11)& 2 \cdot  16 = 32& 7 \cdot  14 = 98& 11 \cdot  17 = 187& 187& 32\\
 \hline
1_3(12)& 2 \cdot  16 = 32& 7 \cdot  17 = 119& 11 \cdot  14 = 154& 154& 32\\
 \hline
1_3(13)& 2 \cdot  17 = 34& 7 \cdot  11 = 77& 14 \cdot  16 = 224& 224& 34\\
 \hline
1_3(14)& 2 \cdot  17 = 34& 7 \cdot  14 = 98& 11 \cdot  16 = 176& 176& 34\\
 \hline
(\#) 1_3(15)& 2 \cdot  17 = 34& 7 \cdot  16 = 112& 11 \cdot  14 = 154& 154& 34\\
 \hline   
   \end{array}
$$

\bigskip
\bigskip
{\bf Acknowledgment:} The author would be grateful to Prof. T. Yamaguchi for inquiring whether this interesting problem (Proposition \ref{PP1}) has an affirmative solution or not.

\vspace*{10mm}

Department of Mathematics, 

Faculty of Education, 

Kochi University 
 
\bigskip
\bigskip

Akebono-cho 2-5-1, Kochi\ 780-8520, 

JAPAN

ssmoda@kochi-u.ac.jp

\end{document}